\documentclass[12pt]{article}
\usepackage{indentfirst,enumerate,cite,amssymb,amsfonts,amsmath,amsthm,mathrsfs,dsfont}
\usepackage{color}
\usepackage{comment}
\usepackage[colorlinks=true,linkcolor=blue,citecolor=blue]{hyperref}
\usepackage{enumerate}
\usepackage{geometry}
\usepackage[utf8]{inputenc}

\numberwithin{equation}{section}

\newcommand{\N}{{\mathbb{N}}}
\newcommand{\Z}{\mathbb{Z}}

\newcommand{\R}{\mathbb{R}}
\newcommand{\C}{\mathbb{C}}
\newcommand{\TT}{{\mathbb{T}}}
\newcommand{\LL}{{\mathbb{L}}}

\newcommand{\T}{{\mathbb{T}}}

\theoremstyle{plain}
\newtheorem{theorem}{Theorem}[section]

\newtheorem{proposition}[theorem]{Proposition}
\newtheorem{corollary}[theorem]{Corollary}
\newtheorem{definition}[theorem]{Definition}

\theoremstyle{definition}

\newtheorem{remark}[theorem]{Remark}

\numberwithin{equation}{section}

\begin{document}
	
\title{Systems of differential operators in \\ time-periodic Gelfand-Shilov spaces}
	
	\date{}		

\author{
	Fernando de \'Avila Silva, 
	Marco Cappiello, 
	Alexandre Kirilov 
}

	\maketitle

\begin{abstract}
This paper explores the global properties of time-independent systems of operators in the framework of Gelfand-Shilov spaces. Our main results provide both necessary and sufficient conditions for global solvability and global hypoellipticity, based on analysis of the symbols of operators. We also present a class of time-dependent operators whose solvability and hypoellipticity are linked to the same properties of an associated time-independent system, albeit with a loss of regularity for temporal variables.	
\end{abstract}

\maketitle


\section{Introduction}

This paper aims to discuss existence and regularity of solution for systems of the type
\begin{equation}\label{general-const-system}
	L_r u=f_r, \qquad r=1,\ldots,\ell,
\end{equation} with
\begin{equation}\label{general-const-system2}	L_r = Q_r(D_t)+ d_rP(x,D_x), \qquad r=1,\ldots,\ell,
\end{equation}
where $d_r \in \C$, and $f_r$ are assigned functions or distributions. The operator
$$
Q_r(D_t)=\sum_{|\alpha|\leq k_r} c_{\alpha,r} D_t^\alpha
$$
is a constant coefficient differential operator of order $k_r$ on the torus $\TT^m$, with $\alpha= (\alpha_1, \ldots, \alpha_n) \in \mathbb{N}_0^m$,  $c_{\alpha,r} \in \C,$ and $D_t^\alpha = D_1^{\alpha_1} D_2^{\alpha_2}\cdots D_m^{\alpha_m}$, whereas $P(x,D_x)$
is a normal differential operator of the form
\begin{equation}\label{P-intro}
	P = P(x,D) = \sum_{|\alpha| + |\beta| \leq M} c_{\alpha,\beta} x^{\beta} D_x^{\alpha}, \ c_{\alpha,\beta} \in \C,
\end{equation}
of order $M\geq 2$, satisfying  the global ellipticity property 
\begin{equation}\label{P-elliptic}
	p_M(x,\xi) = \sum_{|\alpha| + |\beta| = M} c_{\alpha,\beta} x^{\beta} \xi^{\alpha} \neq 0, \quad (x,\xi) \neq (0,0).
\end{equation}

The normal condition and property \eqref{P-elliptic} imply that $P$ has a discrete spectrum consisting in a sequence of real eigenvalues $\lambda_j$ such that $|\lambda_j| \to +\infty$ for $j \to +\infty.$ The corresponding eigenfunctions $\varphi_j$ form an orthonormal basis of $L^2(\R^n)$. Moreover, the following asymptotic Weyl formula holds
\begin{equation}\label{weyl}
	|\lambda_j| \sim  \rho j^{M/2n}, \ \textrm{ as } \ j \to \infty,
\end{equation}
for some positive constant $\rho$. 

The most relevant example is the Harmonic oscillator $P(x, D_x)= -\Delta+|x|^2, $ where $\Delta$ denotes the standard Laplace operator on $\R^n$, and its powers $(-\Delta+|x|^2)^M$, with $M \in \N, M \geq 2.$ Operators of the form \eqref{P-intro}, \eqref{P-elliptic} and their pseudo-differential generalizations have been first studied on the Schwartz space $\mathscr{S}(\R^n)$ of smooth rapidly decreasing functions and on the dual space of tempered distributions $\mathscr{S}'(\R^n)$, cf. \cite{Shu87}. More recently, it was noticed that solutions to equations involving these operators possess stronger regularity and decay properties. For instance, the eigenfunctions of the Harmonic oscillator are the well-known Hermite functions that admit a Gaussian decay at infinity, they are analytic on $\R^n$ and can be extended to entire functions on $\C^n$. A more natural functional setting is then given by the Gelfand-Shilov spaces of type $\mathscr{S}$ introduced in \cite{GelShi67book3, GelShi68book2} and defined as follows.

Given $\mu >0, \nu >0$, the Gelfand-Shilov space $\mathcal{S}^\mu_\nu(\R^n)$ is defined as the space of all $f \in C^\infty(\R^n)$ such that
\begin{equation*}
	\sup_{\alpha, \beta \in \N^n} \sup_{x \in \R^n} A^{-|\alpha+\beta|} \alpha!^{-\nu}\beta!^{-\mu}
	|x^\alpha \partial_x^\beta f(x)| <+\infty 
\end{equation*}
for some $A>0$, or equivalently,
\begin{equation*}
	\sup_{\beta \in \N^n} \sup_{x \in \R^n} C^{-|\beta|} \beta!^{-\mu}\exp( c|x|^{1/\nu})
	| \partial_x^\beta f(x)| <+\infty 
\end{equation*}
for some $C,c>0$.
Elements of $\mathcal{S}^\mu_\nu(\R^n)$ are then smooth functions presenting uniform analytic or Gevrey estimates on $\R^n$ and admitting an exponential decay at infinity. In particular, for $\mu <1$ a function $f \in \mathcal{S}^\mu_{\nu}(\R^n)$  admits an entire extension $\tilde{f}: \C^n \to \C$ satisfying exponential estimates of the form
$$|\partial_z^\beta \tilde f(z) | \leq C^{|\beta|+1} \beta!^\mu \exp \left( -a |\Re z |^{\frac1{\nu}}+ b |\Im z|^{\frac1{1-\mu}}\right), \ z \in \C^n,$$
for some positive constants $C, a,b$ independent of $\beta$. We also recall that $\mathcal{S}^\mu_\nu(\R^n) \neq \{0\}$ if and only if $\mu +\nu \geq 1.$
The elements of the dual space $(\mathcal{S}^\mu_\nu)'(\R^n)$ are commonly known as \textit{temperate ultradistributions}, cf. \cite{Pil88}.

In the last two decades, Gelfand-Shilov spaces have been employed frequently in the study of microlocal and time-frequency analysis with many applications to partial differential equations, see for instance \cite{Ari24,AriAscCap22,AscCap08,AscCap19,CapGraRod06,CapGraRod10jam,CapGraRod10cpde,CapPilPra16,CorNicRod15rmi,CorNicRod15tams,NicRod10book,Pra13} and the references quoted therein. 
Concerning in particular the operators in \eqref{P-intro}, \eqref{P-elliptic}, we mention the hypoellipticity results in \cite{CapGraRod06,CapGraRod10jam} which show that the solutions $u \in \mathscr{S}'(\R^n)$ of the equation $Pu=f \in \mathcal{S}_\nu^\mu(\R^n)$ actually belong to $\mathcal{S}_\nu^\mu(\R^n)$. In particular, the eigenfunctions of $P$ are in $\mathcal{S}^{1/2}_{1/2}(\R^n)$. Recently, these spaces have been also characterized in terms of eigenfunction expansions, see \cite{CapGraPilRod19,GraPilRod11}. Given the above considerations, it is natural to study also systems of the form \eqref{general-const-system}, \eqref{general-const-system2} in the frame of Gelfand-Shilov spaces.
For simplicity, we shall limit ourselves to consider symmetric spaces $\mathcal{S}_\mu^\mu(\R^n)$ with $\mu \geq 1/2$. To use these spaces in the periodic setting defined by \eqref{general-const-system}, \eqref{general-const-system2} we need to consider and adapt a variant of these spaces introduced in \cite{AviCap22} for the case of the one-dimensional torus. Then, we shall consider the space $\mathcal{S}_{\sigma,\mu}(\TT^m \times \R^n)$ with $\sigma \geq 1, \mu \geq 1/2$, ($\mathcal{S}_{\sigma,\mu}$ in short), as the space of all smooth functions on $\TT^m \times \R^n$ such that
\begin{equation}\label{firstnorm}
	|u|_{\sigma, \mu, C}:=
	\sup_{\alpha, \beta\in \Z^n_+, \gamma \in \Z^m_+}C^{-|\alpha|-|\beta|-|\gamma|}\gamma!^{-\sigma} (\alpha!\beta!)^{-\mu}\sup_{(t,x) \in \mathbb{T}^m\times \R^n} |x^\alpha \partial_x^\beta \partial_t^\gamma u(t,x)|
\end{equation}
is finite for some positive constant $C$, and we study the global hypoellipticity and solvability of system \eqref{general-const-system} in this setting. Notice that the elements of $\mathcal{S}_{\sigma, \mu}(\TT^m \times \R^n)$   belong to $\mathcal{S}_{\mu}^{\mu} (\R^n)$, cf. \cite{GelShi67book3,GelShi68book2}, with respect to the variable $x$, while are Gevrey regular on $t \in \TT^m$. The first step is to give a characterization of this space in terms of eigenfunction expansions adapting an analogous result proved in \cite{CapGraPilRod19,GraPilRod11} for classical Gelfand-Shilov spaces and in \cite{AviCap22} for time-periodic Gelfand-Shilov spaces on $\TT \times \R^n$. Namely, the orthonormal basis of eigenfunctions $\{\varphi_j \}_{j \in \N}$ of $P$ allows to write any $u \in \mathcal{S}_{\sigma,\mu}$ (respectively $u \in \mathcal{S}'_{\sigma,\mu}$ ) as the sum of a Fourier series 
\begin{equation*}
	u(t,x) = \sum_{j \in \N} u_j(t) \varphi_j(x),
\end{equation*} where $u_j(t)$ is a sequence of Gevrey functions (respectively distributions) on the $\TT^m$ satisfying suitable exponential estimates.
This allows to discretize the equations $L_r u=f_r$ and to apply 
the typical arguments of the analysis of global hypoellipticity and solvability on the torus.
We note that the study of systems of periodic differential operators has been receiving significant attention more recently. This can be seen in the references \cite{AviMed21, BerKir07, BerKirNunZan12, BerMedZan12, BerMedZan21}. Some of our results and techniques of proof were inspired by ideas used in articles such as \cite{AriKirMed19, AriLes21, Les21}, as well as results developed in Lie groups, see \cite{KirMorRuz20, KirMorRuz21jfa,KirMorRuz21zaa,KirMorRuz22}.
\medskip

After these preliminaries, let us go back to the system \eqref{general-const-system}, \eqref{general-const-system2}. Denoting 
$$
\mathbb{L}  \doteq (L_1, L_2, \ldots, L_\ell),
$$
and fixed $f_1, \ldots, f_\ell \in \mathcal{S}_{\sigma,\mu}$ (or in $\mathcal{S}'_{\sigma,\mu}$), we shall consider the vector-valued function (or distribution)  $F \doteq (f_1,f_2,\ldots,f_\ell)$. Then we can write our system as
$$\LL u(t,x)=F(t,x)$$
or by components as
$$ L_ru=f_r, \ \ r=1,2,\ldots,\ell.$$
\indent

For every $r \in \{1,\ldots, \ell \}$ the symbol of the operator $L_r$ is defined by
\begin{equation*}
	\sigma_{L_r}(\tau,j) = \sum_{|\alpha|\leq k_r} c_{\alpha,r} \tau^\alpha + d_r\lambda_j, \quad r=1,\ldots,\ell,
\end{equation*}
where $\tau=(\tau_1,\tau_2,\ldots,\tau_m)\in\mathbb{Z}^m$, $j\in\mathbb{N}$, and $r=1,\ldots,\ell$. Furthermore, we denote the symbol of the system $\mathbb{L}$ as
\begin{equation*}
	\sigma_{\mathbb{L}}(\tau,j) = (\sigma_{L_1}(\tau,j), \sigma_{L_2}(\tau,j), \ldots, \sigma_{L_\ell}(\tau,j)),
\end{equation*}
and we define
$$
\|\sigma_{\mathbb{L}}(\tau,j)\| \doteq \max_{1\leq r\leq \ell} |\sigma_{L_r}(\tau,j)|.
$$

Notice that the operators $L_r$ defined in \eqref{general-const-system} commute, i.e., for all $1\leq r,s\leq \ell$,
\begin{align*}
	[L_r,L_s] &= L_r L_s - L_s L_r = 0.
\end{align*}
	
	Thus, if $u=u(t,x)$ is a solution to the equations $L_ru=f_r$, for $r=1,\ldots, \ell$, then 
	$L_r f_s = L_rL_s u=L_sL_r u=L_s f_r,$ for all $1\leq r,s\leq \ell$. Therefore, for the system \eqref{general-const-system} to have a solution, it is necessary to ensure that 
	\begin{equation}\label{comp-cond-system}
		L_r f_s = L_s f_r, \quad \text{for } 1\leq r,s\leq \ell. 
	\end{equation}
	
	Moreover, If $u=u(t,x)$ is a solution of the system $L_ru=f_r$, $1\leq r\leq \ell$, then its Fourier coefficients satisfy
	$$ \sigma_{L_r}(\tau,j) \widehat{u_j}(\tau)=\widehat{f_{r,j}}(\tau), \quad 1\leq r \leq \ell, $$
	for all $\tau \in \mathbb{Z}^n$ and $j\in \mathbb{N}$. 
	Therefore, for each $r=1,\ldots,\ell$, we have
	\begin{equation}\label{comp-cond-op}
		\sigma_{L_r}(\tau,j)=0 \implies \widehat{f_{r,j}}(\tau) = 0.
	\end{equation}

By combining the two remarks above, we define the set
\begin{equation*}
	\mathscr{E}_\LL = \left\{F=(f_1, \ldots,f_\ell): f_r \in \mathcal{S}_{\sigma,\mu}' \mbox{ satisfies \eqref{comp-cond-system} and \eqref{comp-cond-op} for all } 1\leq r \leq \ell  \right\}.
\end{equation*}

Note that, if $F\not \in \mathscr{E}_\LL$, then there is no $u\in \mathcal{S}_{\sigma,\mu}'$ such that $\LL u = F$. The set $\mathscr{E}_\LL$ is called the set of admissible ultradistributions for the solvability of the system $\LL$.

\begin{definition}\label{GS-GH-definition}
	Given a system $\LL=(L_1, L_2, \ldots, L_\ell)$ in the form \eqref{general-const-system}, we say that:
	\begin{enumerate}
		\item $\LL$ is $\mathcal{S}_{\sigma,\mu}$-globally solvable when  $\LL(\mathcal{S}_{\sigma,\mu}')=\mathscr{E}_\LL$;
		\item $\LL$ is $\mathcal{S}_{\sigma,\mu}$-globally hypoelliptic when $u \in \mathcal{S}_{\sigma,\mu}'$  and $L_r u \in \mathcal{S}_{\sigma,\mu},$ for all $1\leq r\leq\ell$ imply  $u \in \mathcal{S}_{\sigma,\mu}.$
	\end{enumerate}
\end{definition}

We are now prepared to present the main results of the paper. Regarding the regularity of solutions for time-independent systems, let us denote the set of zeros of the symbol as:
\begin{equation*}
	\mathcal{N} = \{(\tau,j)\in\mathbb{Z}^m\times\mathbb{N}: \sigma_{\mathbb{L}}(\tau,j)=0 \},
\end{equation*}
and introduce the following result:

\begin{theorem} \label{GH-thm}
	The system $\LL$ defined by \eqref{general-const-system} is $\mathcal{S}_{\sigma,\mu}$-globally hypoelliptic if and only if 
	$\mathcal{N}$ is finite and for all $\varepsilon>0$, there exists $C_\varepsilon>0$ such that  
	$$\|\sigma_{\mathbb{L}}(\tau,j)\| \geq C_\varepsilon \exp\left[-\varepsilon(\|\tau\|^{1/\sigma}+j^{1/(2n\mu)})\right],$$
	for any $(\tau,j)\in\Z^m\times\N$ such that $\sigma_{\mathbb{L}}(\tau,j)\neq0.$
\end{theorem}

Concerning global solvability, we have the following result. 

\begin{theorem}\label{GS-thm}
	The system $\LL$ is $\mathcal{S}_{\sigma,\mu}$-globally solvable on $\TT^m\times \R^n$ if and only if for all $\varepsilon>0$, there exists $C_\varepsilon>0$ such that 
	\begin{equation*}
		\|\sigma_{\mathbb{L}}(\tau,j)\| \geq C_\varepsilon \exp\left[-\varepsilon(\|\tau\|^{1/\sigma}+j^{1/(2n\mu)})\right]
	\end{equation*}
	for all $(\tau,j)\in\Z^m\times\N$ such that $\|\sigma_{\mathbb{L}}(\tau,j)\|\neq 0,$
\end{theorem}

Observe the condition for achieving $\mathcal{S}_{\sigma,\mu}$-global solvability in the previous theorem coincides with one of the requirement for obtaining $\mathcal{S}_{\sigma,\mu}$-global hypoellipticity.  The distinction lies in the fact that the global solvability problem can be addressed also in cases where the set $\mathcal{N}$ contains infinitely many elements. Consequently, we arrive at the following corollary.

\begin{corollary}
	If $\LL$ is $\mathcal{S}_{\sigma,\mu}$-globally hypoelliptic, then it is also $\mathcal{S}_{\sigma,\mu}$-globally solvable.	
\end{corollary} 

The paper is structured as follows: In Section 2, we define time-periodic Gelfand-Shilov spaces, given in \cite{AviCap22}, clarifying their characterization through eigenfunction expansions. Next, we give a comprehensive characterization of functions and ultradistributions based on their Fourier coefficients with respect to all variables on $\TT^m\times \R^n$ simultaneously. In Section 3, we prove Theorem \ref{GH-thm}, which provides a complete characterization of global hypoellipticity for systems of the form \eqref{general-const-system}. Following this, Section 4 presents the proof of Theorem \ref{GS-thm}, focusing on global solvability for ultradistributions. We demonstrate that this class of systems admits strong solutions, and we highlight the relationship between global hypoellipticity and global solvability. In Section 5, we introduce examples involving Liouville vectors, a phenomenon commonly encountered in fully periodic systems. Finally, Section 6 introduces a class of systems where the global hypoellipticity and solvability are intertwined with a related time-independent system.

\section{Notations and preliminary results  \label{sec-notations}} 

Let us begin by recalling the definition and some fundamental properties of the spaces $\mathcal{S}_{\sigma, \mu}(\TT^m \times \R^n)$ and $\mathcal{S}'_{\sigma, \mu}(\TT^m \times \R^n)$.

\subsection{Time-periodic Gelfand-Shilov spaces and eigenfunction expansions}

Throughout this paper, we denote by $\mathcal{G}^{\sigma,h}(\TT^m)$, where $h > 0$ and $\sigma \geq 1$, the space of all smooth functions $\varphi \in C^{\infty}(\TT^m)$ such that there exists a constant $C > 0$ satisfying 
\begin{equation*}
	\sup_{t \in \TT^m} |\partial^{\gamma}\varphi(t)| \leq C h^{\gamma}(\gamma!)^{\sigma}, \quad \forall \gamma \in \Z_+^m.
\end{equation*}

Thus, $\mathcal{G}^{\sigma,h}(\TT^m)$ is a Banach space equipped with the norm
\begin{equation*}
	\|\varphi\|_{\sigma,h} = \sup_{\gamma \in \Z_+^n}\left \{\sup_{t \in \TT^m} |\partial^{\gamma}\varphi(t)|  h^{-\gamma}(\gamma!)^{-\sigma}\right\},
\end{equation*}
and the space of periodic Gevrey functions of order $\sigma$ is defined by
$$
\mathcal{G}^{\sigma}(\TT^m) = \underset{h\rightarrow +\infty}{\operatorname{ind} \lim} \ \mathcal{G}^{\sigma, h} (\TT^m).
$$
Its dual space is denoted by $(\mathcal{G}^{\sigma})'(\TT^m)$.

Recall that the Fourier coefficients of a Gevrey function (ultradistribution) $f$ on $\TT^m$ are defined as
\begin{equation*}
	\widehat{f}(\tau) = \frac{1}{(2\pi)^m} \int_{\TT^m} f(t) e^{-i \tau \cdot t} dt, \quad \tau \in \Z^m,
\end{equation*}
with the dual correspondence 
\begin{equation*}
	\widehat{f}(\tau) = \langle f, e^{i \tau \cdot t} \rangle, \quad \tau \in \Z^m,
\end{equation*}
for $f \in (\mathcal{G}^{\sigma})'(\TT^m)$. 

In particular, functions and distributions can be represented in terms of Fourier series as
\begin{equation*}
	f(t) = \sum_{\tau \in \Z^m} \widehat{f}(\tau) e^{i \tau \cdot t}.
\end{equation*}

\begin{theorem}\label{Theorem_char_G}
	Let $\{a_\tau\}_{\tau \in \Z^{m}}$ be a sequence of complex numbers, and consider the formal series
	\begin{equation}\label{formal_series}
		f(t) = \sum_{\tau\in \Z^m} a_\tau e^{i \tau \cdot t}, 
	\end{equation}
	where $t\in\TT^m$. Then the series \eqref{formal_series}:
	\begin{itemize}
		\item[(a)] converges in $\mathcal{G}^{\sigma}(\TT^m)$ if and only if there exist $\varepsilon>0$ and $C>0$ such that
		\begin{equation}\label{seq-partial-coeff-functions}
			|a_{\tau}| \leq C e^{-\varepsilon\|\tau\|^{\frac{1}{\sigma}}}, \quad \text{for all } \tau \in \Z^m. 
		\end{equation}
		
		\item[(b)] converges in $(\mathcal{G}^{\sigma})'(\TT^m)$ if and only  for all $\varepsilon>0$, there exists $C_\varepsilon>0$ such that
		\begin{equation*}
			|a_{\tau}| \leq C_\varepsilon e^{\varepsilon\|\tau\|^{\frac{1}{\sigma}}}, \quad \text{for all } \tau \in \Z^m.
		\end{equation*}
	\end{itemize}
	
Moreover, if any of the convergence conditions above hold, then for each $\tau \in \Z^m$, we have $\widehat{f} (\tau) = a_\tau$, for all $\tau \in \Z^m.$
\end{theorem}

Fixing $\sigma \geq 1$, $\mu\geq 1/2$, and $C>0$, and denoting by $\mathcal{S}_{\sigma, \mu, C}$ the space of all smooth functions on $\TT^m \times \R^n$ for which the norm \eqref{firstnorm} is finite, it is straightforward to prove that $\mathcal{S}_{\sigma, \mu, C}$ is a Banach space. We can then equip $\mathcal{S}_{\sigma, \mu} = \bigcup\limits_{C> 0} \mathcal{S}_{\sigma,\mu,C}$ with the inductive limit topology given by
$$
\mathcal{S}_{\sigma,\mu}  =\displaystyle \underset{C\rightarrow +\infty}{\mbox{ind} \lim} \;\mathcal{S}_{\sigma,\mu,C}.
$$
We denote by $\mathcal{S}_{\sigma, \mu}'$ the space of all linear continuous forms $u: \mathcal{S}_{\sigma, \mu} \to \C$.

\begin{remark}
	According to \cite[Lemma 3.1]{GraPilRod11}, an equivalent norm to \eqref{firstnorm} on $\mathcal{S}_{\sigma,\mu,C}$ is given by:
	\begin{equation}\label{secondnorm}
		\| u \|_{\sigma, \mu,C} := \sup_{\gamma \in \Z^m_+, M \in \N} C^{-M-|\gamma|} M!^{-m\mu} \gamma!^{-\sigma} \| P^M \partial_t^\gamma u \|_{L^2(\TT^m \times \R^n)}.
	\end{equation}
	In order to take advantage of the good properties of the operator $P(x,D_x)$ we shall often use the norm \eqref{secondnorm} instead of \eqref{firstnorm}.
\end{remark}

Now we want to recover the Fourier analysis presented in \cite{AviCap22}. With this purpose we recall the characterization of $\mathcal{S}_{\sigma, \mu}$ and $\mathcal{S}'_{\sigma, \mu}$ in terms of eigenfunction expansions. For this, let $\varphi_j \in \mathcal{S}_{1/2}^{1/2}(\R^n), j \in \N$, be the eigenfunctions of the operator $P$ in \eqref{P-intro}. We have the following results.

In the sequel we recall a characterization of $\mathcal{S}_{\sigma,\mu}(\mathbb{T}^m\times \R^n)$ and 
$\mathcal{S}_{\sigma,\mu}'(\mathbb{T}^m\times \R^n)$ in terms of the Fourier analysis generated by $\{\varphi_j \}$. First, we observe that if $u \in \mathcal{S}_{\sigma,\mu}' (\mathbb{T}^m\times \R^n)$, then for any $j \in \Z_+$ the linear form $u_j: \mathcal{G}^{\sigma} \to \C$ given by
\begin{equation}\label{partial}
	\langle  u_j(t) \, , \, \psi(t) \rangle \doteq 
	\langle  u ,  \psi(t)\varphi_j(x) \rangle
\end{equation}
belongs to $(\mathcal{G}^{\sigma})'(\TT^m)$. Moreover, for every $\varepsilon>0$ and $h >0$, there exists $C_{\varepsilon,h} >0$ such that
\begin{equation*} 
	|\langle  u_j(t) \, , \, \psi(t) \rangle | \leq C_{\varepsilon,h} \| \psi\|_{\sigma, h} \exp \left(\varepsilon j^{\frac1{2n\mu}}\right),
\end{equation*}
for all $j \in \N,$ and $\psi \in \mathcal{G}^{\sigma,h}(\TT^m)$.

Then, we can express
\begin{equation*}
	\langle  u \, , \, \theta \rangle  =
	\sum_{j \in \N} \langle u_j(t)\varphi_j(x)  \, , \, \theta \rangle
\end{equation*}
where
\begin{equation*}
	\langle  u_j(t)\varphi_j(x) \, , \, \theta \rangle \doteq 
	\left\langle u_j(t)  \, , \, \int_{\mathbb{R}^n} \theta(t,x) \varphi_j(x) dx \right\rangle 
\end{equation*}
and $\{u_j(t)\}_{j \in \N}$ is a sequence in $(\mathcal{G}^{\sigma})'(\TT^m)$ as defined in \eqref{partial}.

Conversely, given a sequence $\{u_j(t)\}_{j \in \N}$ in $(\mathcal{G}^{\sigma})'(\TT^m)$ such that for every $\varepsilon>0$ and $h>0$, there exists $C_{\varepsilon, h}>0$ satisfying
\begin{equation}\label{estimate-distr}
	|\langle  u_j \, , \, \psi \rangle| \leq C_{\varepsilon, h} \|\psi\|_{\sigma,h}
	\exp\left(\varepsilon j^{\frac{1}{2n\mu}}\right),
\end{equation}
for all $j \in \N$ and $\psi \in \mathcal{G}^{\sigma,h}(\TT^m)$.

Then $u = \sum_{j \in \N} u_j(t) \varphi_j(x)$ 
belongs to $\mathcal{S}_{\sigma,\mu}' (\mathbb{T}^m\times \R^n)$, and
$$
\langle  u_j \, , \, \psi(t) \rangle  = \langle  u \, , \, \psi(t)\varphi_j(x) \rangle.
$$
for every $\psi \in  \mathcal{G}^{\sigma}(\TT^m)$.

In particular, if $\{u_j\}_{j \in \N}$ is a sequence of functions in $\mathcal{G}^{\sigma}(\TT^m)$, then we can replace \eqref{estimate-distr} by
\begin{equation*}
	\sup_{t \in \TT^m} | u_j(t)|\leq C_{\varepsilon} 
	\exp\left(\varepsilon j^{\frac{1}{2n\mu}}\right),
\end{equation*}
for all $j \in \N.$

\begin{theorem}\label{Thm-estimate-derivatives-functions}
	Let $\mu\geq 1/2$ and $\sigma \geq 1$, and let $u \in \mathcal{S}_{\sigma,\mu}'$. Then $u \in \mathcal{S}_{\sigma,\mu}$ if and only if it can be represented as 
	\begin{equation*}
		u(t,x) = \sum_{j \in \N} u_j(t) \varphi_j(x),
	\end{equation*}
	where
	$$
	u_j(t) = \int_{\R^n} u(t,x)\varphi_j(x)dx,
	$$
	and there exist $C >0$ and $\varepsilon>0$ such that
	\begin{equation} \label{deccoeff}
		\sup_{t \in \TT^m} | \partial_t^\gamma u_j(t)| \leq
		C^{|\gamma|+1} (\gamma!)^{\sigma} \exp \left[-\varepsilon j^{\frac{1}{2n\mu}} \right], \quad \forall j \in \N, \forall \gamma \in \Z^m_+.
	\end{equation}		
	
\end{theorem}
\begin{proof}
	See	\cite[Theorem 2.4]{AviCap22}.
\end{proof}

\subsection{Full Fourier coefficients}

In this section, we provide the characterization of $\mathcal{S}_{\sigma,\mu} (\mathbb{T}^m\times \R^n)$ and $\mathcal{S}_{\sigma,\mu}' (\mathbb{T}^m\times \R^n)$ in terms of Fourier coefficients with respect to all variables simultaneously, referred to here as the full Fourier coefficients of the elements in these spaces. 

\begin{theorem}\label{charac_full_fourier-functions}
	Let $\{a(\tau,j)\}_{(\tau,j) \in \Z^m \times \N}$ be a sequence of complex numbers, and consider the formal series
	$$
	a(t,x) = \sum_{j \in \N} \sum_{\tau \in \Z^m} a(\tau,j) e^{it\cdot \tau}\varphi_j(x),
	$$
	where $(t,x)\in\TT^m\times\R^n$.
	
	Then, $a \in \mathcal{S}_{\sigma,\mu}$ if and only if there exist $\varepsilon>0$ and $C>0$ such that
	\begin{equation}\label{seq-full-coeff-funct}
		|a(\tau,j)| \leq C 
		\exp\left[-\varepsilon \left(\|\tau\|^{\frac{1}{\sigma}} + j^{\frac{1}{2n\mu}}\right)\right], \ 
		\forall (\tau,j) \in \Z^m\times\N.
	\end{equation}
	
	Moreover, when this condition holds, we have $\widehat{a_j}(\tau) = a(\tau,j)$ for all $(\tau,j) \in \Z^m\times \N$, where
	$$
	\widehat{a_j}(\tau)=\int_{\TT^m}a_j(t)e^{it\cdot\tau}dt,  \quad \text{and} \quad  a_j(t) = \int_{\R^n} a(t,x)\varphi_j(x)dx.
	$$
\end{theorem}

\begin{proof}
	Assume that the sequence $\{a(\tau,j)\}$ satisfies the condition \eqref{seq-full-coeff-funct}. It follows from item $(a)$ of Theorem \ref{Theorem_char_G} that, for each $j\in \N$, the function 
	$$
	a_j(t) = \sum_{\tau \in \Z^m} a(\tau,j) e^{it \cdot \tau} \in \mathcal{G}^{\sigma}(\TT^m).
	$$
	
	Now, for each fixed $\gamma \in \Z^m_+$, by using standard estimates and \eqref{seq-full-coeff-funct}, we obtain
	$$
	|\partial^\gamma a_j(t)|  \leq 
	C \exp\left(-\varepsilon j^{1/2n\mu}\right) 
	\sum_{\tau \in \Z^m} \|\tau\|^{|\gamma|}\exp\left(-\varepsilon \|\tau\|^{1/\sigma}\right).
	$$
	
	Since, for every $\tau \in \Z^m$, we have
	$$
	\|\tau\|^{|\gamma|}\exp\left(-\varepsilon/2 \|\tau\|^{1/\sigma}\right) \leq C_{\varepsilon}^{|\gamma|} (\gamma!)^{\sigma},
	$$
	then
	$$
	|\partial^\gamma a_j(t)|  \leq 
	C C_{\varepsilon}^{|\gamma|} C_1 \exp\left(-\varepsilon/2 j^{1/2n\mu}\right), 
	$$
	where 
	$C_1 = \sum_{\tau \in \Z^m}\exp\left(-\varepsilon/2 \|\tau\|^{1/\sigma}\right)$. Therefore, $a \in \mathcal{S}_{\sigma,\mu}$ in view of \eqref{deccoeff}.
	
	Conversely, if $a \in \mathcal{S}_{\sigma,\mu}$, then for $\gamma \in \N^m$ we have
	$$
	|\widehat{a_j}(\tau)|\leq C \dfrac{C^{|\gamma|} (\gamma !)^{\sigma}}{|\tau^{\gamma}|}  \exp\left(-\varepsilon j^{1/2n\mu}\right), \forall \tau\neq 0.
	$$
	Taking the infimum on $|\gamma|$ of the right-hand side, we obtain the result.
\end{proof}

\begin{theorem}\label{charac_full_fourier_ultradistr}
	Let $\{a(\tau,j): (\tau,j) \in \Z^m \times \N\}$ be a sequence of complex numbers, and consider the formal series
	$$
	a(t,x) = \sum_{j \in \N} \sum_{\tau \in \Z^m} a(\tau,j) e^{it\cdot \tau}\varphi_j(x),
	$$
	where $(t,x)\in\TT^m\times\R^n$.
	
	Then, $a \in \mathcal{S}_{\sigma,\mu}'$ if and only if for every $\varepsilon>0$, there exists $C_{\varepsilon}>0$ such that
	\begin{equation}\label{seq-full-coeff-distr}
		|a(\tau,j)| \leq C_{\varepsilon} 
		\exp\left[\varepsilon \left(\|\tau\|^{\frac{1}{\sigma}} + j^{\frac{1}{2n\mu}}\right)\right], \ 
		\forall \tau \in \Z^m.
	\end{equation}
	
	Moreover, when this condition holds, we have $ a(\tau,j)=\widehat{a_j}(\tau)$ for all $(\tau,j) \in \Z^m\times \N$, where
	$$
	\langle  a_j(t) \, , \, \psi(t) \rangle \doteq 
	\langle  a ,  \psi(t)\varphi_j(x) \rangle, \quad \psi \in \mathcal{G}^{\sigma}(\TT^m).
	$$
\end{theorem}

\begin{proof}
	Assume first that $\{a(\tau,j)\}$ satisfies the condition \eqref{seq-full-coeff-distr}. Then, it follows from item $(b).$ of Theorem  \ref{Theorem_char_G} that 
	$$
	a_j(t) = \sum_{\tau \in \Z^m} a(\tau,j) e^{it \cdot \tau} \in \mathcal{G}^{\sigma}(\TT^m), \quad j \in \N.
	$$ 
	
	For each $\ell\in\N$, define $A_\ell\in \mathcal{S}_{\sigma,\mu}'$ by 
	$$
	A_\ell=  \sum_{k=1}^{\ell} a_k(t)\varphi_k(x),
	$$
	and fix $\theta(t,x) \in \mathcal{S}_{\sigma,\mu}$. Note that
	\begin{align*}
		\langle A_{\ell+\eta} - A_{\eta}, \theta \rangle & = 
		\sum_{k=\eta+1}^{\ell}  \sum_{\tau \in \Z^m} \langle  a(\tau,k) \exp (it \cdot \tau) \varphi_k(x), \theta \rangle \\
		& = \sum_{k=\eta+1}^{\ell}  \sum_{\tau \in \Z^m} 
		a(\tau,k) \widehat{\theta}(-\tau,k)
	\end{align*}
	
	Since $\{\widehat{\theta}(-\tau,k)\}$ satisfies condition \eqref{seq-full-coeff-funct}, there exist $\widetilde{\varepsilon}>0$  and  $C>0$ such that
	\begin{equation*}
		|\widehat{\theta}(-\tau,k)| \leq C 
		\exp\left[-\widetilde{\varepsilon} \left(\|\tau\|^{\frac{1}{\sigma}} + j^{\frac{1}{2n\mu}}\right)\right], \quad \tau \in \Z^m.
	\end{equation*}
	
	Then, if we take $\varepsilon = \widetilde{\varepsilon}/2$ in \eqref{seq-full-coeff-distr} we obtain $C_{\widetilde{\varepsilon}} >0$ satisfying
	\begin{equation*}
		|a(\tau,j)| \leq C_{\widetilde{\varepsilon}}
		\exp\left[\widetilde{\varepsilon}/2 \left(\|\tau\|^{\frac{1}{\sigma}} + j^{\frac{1}{2n\mu}}\right)\right], \quad \tau \in \Z^m.
	\end{equation*}
	
	Therefore, 
	$$
	|\langle A_{\ell+\eta} - A_{\eta}, \theta \rangle| \leq
	C C_{\widetilde{\varepsilon}} C_1 
	\sum_{k=\eta+1}^{\ell}  \exp \left(-\widetilde{\varepsilon}/2 \, k^{\frac{1}{2n\mu}}\right)
	$$
	where
	$
	C_1 = \sum_{\tau \in \Z^m} 	\exp \left(-\widetilde{\varepsilon}/2 \, \|\tau\|^{\frac{1}{\sigma}}\right).
	$
	
	Hence, $\{A_\ell (\theta)\}_{\ell\in\N}$ is a Cauchy sequence in $\C$, for all $\theta \in \mathcal{S}_{\sigma,\mu}$, implying that $\{A_\ell\}_{\ell\in\N}$ converges to some $a \in \mathcal{S}_{\sigma,\mu}'$.
	
	Conversely, if $a \in \mathcal{S}_{\sigma,\mu}'$, then
	$$
	a(t,x) = \sum_{j \in \N} \sum_{\tau \in \Z^m} a_j(t) e^{it\cdot \tau} \varphi_j(x)
	$$
	where $a_j(t) \in \mathcal{G}^{\sigma}(\TT^m)$ and satisfy \eqref{estimate-distr}. Note that for every $\varepsilon>0$ and $h>0$ there is $C_{\varepsilon,h}>0$
	such that
	$$
	|\widehat{a}_j(\tau)| = (2\pi)^{-m}|\langle a_j(t) , \exp(-it\tau) \rangle|
	\leq C_{\varepsilon,h} \|\exp(-it\tau)\|_{\sigma,h}\exp(\varepsilon j^{1/2n\mu}),
	$$
	since $t\mapsto \exp(-it\tau)$ belongs to $\mathcal{G}^{\sigma,h}(\TT^m)$, for all $\tau \in \Z^m$ and all $h>0$. In particular, for $h>1$ we get
	\begin{align*}
		\|\exp(-it\tau)\|_{\sigma,h} & =	 \sup_{\gamma \in \Z_+^n}\left \{\sup_{t \in \TT^m} |\partial^{\gamma}\exp(-it\tau)|  h^{-\gamma}(\gamma!)^{-\sigma}\right\} \\
		& \leq  \sup_{\gamma \in \Z_+^n} (\gamma !)^{-\sigma} h^{-|\gamma|} \exp(\varepsilon \|\tau\|^{1/\sigma}) \\
		& \leq \exp(\varepsilon \|\tau\|^{1/\sigma}).
	\end{align*}
	
	Therefore, given $\varepsilon >0$ we obtain $C_{\varepsilon}$ such that
	$$
	|\widehat{a}_j(\tau)| \leq C_{\varepsilon}
	\exp(\varepsilon \|\tau\|^{1/\sigma} + j^{1/2n\mu}),
	$$
	implying that the sequence $\{a(\tau,j)\}$ satisfies the estimate \eqref{seq-full-coeff-distr}, and therefore the proof is complete.
\end{proof}

\section{Global hypoellipticity}

In this section we prove Theorem \ref{GH-thm}. We shall split the proof into two Propositions, one addressing sufficiency and the other necessity.

First, recall that the set of zeros of the symbol is 
\begin{equation*}
	\mathcal{N} = \{(\tau,j)\in\mathbb{Z}^m\times\mathbb{N}; \sigma_{\mathbb{L}}(\tau,j)=0 \}.
\end{equation*}

Note that when the set $\mathcal{N}$ has infinitely many elements, it is possible to construct $u \in \mathcal{S}_{\sigma,\mu}' \setminus \mathcal{S}_{\sigma,\mu}$  solution of the  equation $\mathbb{L} u=0$. Consequently, the system $\mathbb{L}$ is not $\mathcal{S}_{\sigma,\mu}$-globally hypoelliptic. Indeed, consider the following sequence: 
\[
\widehat{u_j}(\tau) = \begin{cases}
	1, & \text{if } \sigma_{\mathbb{L}}(\tau,j)=0, \\
	0, & \text{otherwise}.
\end{cases}
\]
We observe that this sequence corresponds to an element $u\in \mathcal{S}_{\sigma,\mu}'$ but not in $\mathcal{S}_{\sigma,\mu}$, which satisfies $\mathbb{L}_r u=0$, for all $r=1,\ldots,\ell$.

Thus, henceforth we will assume that $\mathcal{N}$ remains finite throughout this paper.

\begin{proposition}\label{Necess2-GH}
	 Assume that $\mathcal{N}$ is finite and that for any $\varepsilon>0$ there exists $C_\varepsilon>0$ such that
	\begin{equation*}
		\|\sigma_{\LL}(\tau,j)\|\geq C_\varepsilon  \exp\left[-\varepsilon(\|\tau\|^{1/\sigma}+j^{1/(2n\mu)}) \right]
	\end{equation*}
	for all  $(\tau,j)\in\Z^m\times\N$ such that $\sigma_{\LL}(\tau,j) \neq 0$. Then $\LL$ is $\mathcal{S}_{\sigma,\mu}$-globally hypoelliptic.
\end{proposition}

\begin{proof}
	Let us start by defining, for each $(\tau,j) \in \mathbb{Z}^m\times \mathbb{N}$, $r^* = r(\tau,j)\in\{1,\ldots,\ell\}$ as the index such that $$|\sigma_{L_{r^*}}(\tau,j)| = \|\sigma_{\mathbb{L}}(\tau,j)\| = \max_{1\leq r\leq \ell} |\sigma_{L_r}(\tau,j)|.$$ 
	
	Now, let $u \in \mathcal{S}_{\sigma,\mu}'$ be a solution of $L_r u =f_r \in \mathcal{S}_{\sigma,\mu}$ for $r=1,\ldots,\ell$. Our goal is to prove that $u \in \mathcal{S}_{\sigma,\mu}$.
	
	Due to the finiteness of the set $\mathcal{N}$, it is sufficient to analyze the behavior of $\widehat{u_j}(\tau)$ outside $\mathcal{N}$. If $(\tau, j) \notin \mathcal{N}$, then $\widehat{u_j}(\tau) = \sigma_{L_{r^*}}(\tau,j)^{-1} \widehat{f_{r^*,j}}(\tau)$, leading to
	\begin{align*}
		|\widehat{u_j}(\tau)| &= |\sigma_{L_{r^*}}(\tau,j)|^{-1}|\widehat{f_{r^*,j}}(\tau)| \\ 
		&= \|\sigma_{\LL}(\tau,j)\|^{-1} |\widehat{f_{r^*,j}}(\tau)| \\
		&\leq C_\varepsilon  \exp\left[\varepsilon (\|\tau\|^{1/\sigma}+j^{1/(2n\mu)}) \right]|\widehat{f_{r^*,j}}(\tau)|.
	\end{align*}
	
	Since $f_{r^*} \in \mathcal{S}_{\sigma,\mu}$, then there is   $\varepsilon_{r^*}>0$ and $C_{r^*}>0$ such that 
	$$ |\widehat{f_{{r^*},j}}(\tau)| \leq C_{r^*} \exp\left[-\varepsilon_{r^*}(\|\tau\|^{1/\sigma}+j^{1/(2n\mu)})\right],
	$$    
	for every $(\tau,j)\in\Z^m\times \N.$
	
	Therefore, for $\varepsilon=\varepsilon_{r^*}/2$ we have
	$$
	|\widehat{u_j}(\tau)|  \leq C_\varepsilon  C_{r^*} \exp\left[-\frac{\varepsilon_{r^*}}{2}(\|\tau\|^{1/\sigma}+j^{1/(2n\mu)})\right].
	$$
	
	Thus, $u \in \mathcal{S}_{\sigma,\mu}$, implying that $\LL$ is $\mathcal{S}_{\sigma,\mu}$-globally hypoelliptic.
\end{proof}

\begin{proposition}\label{Necess2-GH}
	If $\LL$ is $\mathcal{S}_{\sigma,\mu}$-globally hypoelliptic, then for any $\varepsilon>0$, there exists $C_\varepsilon>0$ such that
	\begin{equation*}
		|\sigma_{\LL}(\tau,j)|\geq C_\varepsilon  \exp\left[-\varepsilon(\|\tau\|^{1/\sigma}+j^{1/(2n\mu)}) \right]
	\end{equation*}
	for all  $(\tau,j)\in\Z^m\times\N$ such that $\sigma_{\LL}(\tau,j) \neq 0$.
\end{proposition}
\begin{proof}
	Let us prove this result by contradiction. Assume that there exists $\varepsilon >0$ such that for each $k\in\N$ there exists $(\tau_k,j_k)\in\Z^m\times\N$ satisfying
	\begin{equation*}
		0<\|\sigma_{\LL}(\tau_k,j_k)\| <  \exp\left[-\varepsilon (\|\tau_k\|^{1/\sigma} + j_k^{1/(2n\mu)}) \right].
	\end{equation*}
	
	Now, for each $r=1,\ldots,\ell$, we define 
	$$
	\widehat{f_{r,j}}(\tau)= \begin{cases}
		\sigma_{L_r}(\tau_k,j_k)|(\tau_k,j_k)| , & \text{if } (\tau,j) = (\tau_k,j_k) \text{ for some } k\in \N, \\
		0, & \text{otherwise. }
	\end{cases}
	$$
	
	Note that these functions satisfy the compatibility conditions \eqref{comp-cond-system} and \eqref{comp-cond-op}. Thus,  to prove that $F=(f_1,\ldots,f_\ell) \in \mathscr{E}_\LL$, we have only to show that each component belongs to the space $\mathcal{S}_{\sigma,\mu}$.
	
	Let  $C>0$ be a constant such that for any $ k\in\N$ we have
	\begin{equation*}
		|(\tau_k,j_k)|\exp\left[-\frac{\varepsilon}{2}(\|\tau_k\|^{1/\sigma}+j_k^{1/(2n\mu)}) \right] < C.
	\end{equation*}    
	
	Thus
	\begin{align*}
		|\widehat{f_{r,j}}(\tau)| &= |\sigma_{L_r}(\tau_k,j_k)| |(\tau_k,j_k)| \\[2mm]
		&\leq \|\sigma_{\LL}(\tau_k,j_k)\| |(\tau_k,j_k)| \\
		&\leq  \exp\left[-\varepsilon(\|\tau_k\|^{1/\sigma}+j_k^{1/(2n\mu)}) \right] |(\tau_k,j_k)|  \\[1mm]
		&\leq  \exp\left[-\varepsilon (\|\tau_k\|^{1/\sigma}+j_k^{1/(2n\mu)}) \right] C \exp\left[\frac{\varepsilon}{2}(\|\tau_k\|^{1/\sigma}+j_k^{1/(2n\mu)}) \right]  \\
		&\leq C\exp\left[-\frac{\varepsilon}{2}(\|\tau_k\|^{1/\sigma}+j_k^{1/(2n\mu)}) \right],
	\end{align*}
	then $f_r \in \mathcal{S}_{\sigma,\mu},$ for every $r=1,\ldots,\ell$. 
	
	Now, if we define 
	$$
	\widehat{u_j}(\tau) = 
	\begin{cases}
		|(\tau_k,j_k)| , & \text{if } (\tau,j) = (\tau_k,j_k) \text{ for some } k\in \N, \\
		0, & \text{otherwise. }
	\end{cases}
	$$
	it is clear that   
	$$
	u(t,x)= \sum_{j \in \N}\sum_{\tau \in\Z^m} \widehat{u_j}(\tau)e^{i\tau \cdot t}\varphi_{j}(x) \in \mathcal{S}_{\sigma,\mu}' \setminus  \mathcal{S}_{\sigma,\mu} 
	$$
	and that $\LL u=F$. This proves that $\LL$ is not $\mathcal{S}_{\sigma,\mu}$-globally hypoelliptic.
\end{proof}

\section{Global solvability}

In this section we prove Theorem \ref{GS-thm}. As for the proof of Theorem \ref{GH-thm} the assertion of the theorem will follow combining two propositions, one addressing sufficiency and the other necessity. 

\begin{proposition}\label{GS-implies-DC}
	If $\LL$ is $\mathcal{S}_{\sigma,\mu}$-globally solvable, then  for all $\varepsilon>0$, there exists $C_\varepsilon>0$ such that
	\begin{equation}\label{Dioph-cond-GS-implies-DC}
		\|\sigma_{\mathbb{L}}(\tau,j)\| \geq C_\varepsilon \exp\left[-\varepsilon(\|\tau\|^{1/\sigma}+j^{1/(2n\mu)})\right]
	\end{equation}
	for all $(\tau,j)\in\Z^m\times\N$ whenever $\|\sigma_{\mathbb{L}}(\tau,j)\|\neq 0.$ 
\end{proposition}

\begin{proof} 
	Suppose that \eqref{Dioph-cond-GS-implies-DC} is not true; then there exists $\varepsilon >0$ such that for each $k\in\mathbb{N}$, there exists $(\tau_k,j_k)\in\mathbb{Z}^m\times\mathbb{N}$ satisfying
	\begin{equation*}
		0 < \|\sigma_{\mathbb{L}}(\tau_k,j_k)\| < k^{-1}\exp\left[ -\varepsilon (\|\tau_k\|^{1/\sigma} + j_k^{1/(2n\mu)}) \right].
	\end{equation*}
	
	For each $r=1,\ldots,\ell$, we define 
	$$
	\widehat{f_{r,j}}(\tau)= \begin{cases}
		1 , & \text{if } (\tau,j) = (\tau_k,j_k) \text{ for some } k\in \mathbb{N}, \\
		0, & \text{otherwise. }
	\end{cases}
	$$
	
	It is easy to verify that $F = (f_1, \ldots, f_\ell) \in \mathscr{E}_\mathbb{L}$, meaning that the functions $f_r \in \mathcal{S}_{\sigma,\mu}'$ and satisfy the compatibility conditions \eqref{comp-cond-system} and \eqref{comp-cond-op}.
	
	Now, suppose there exists $u \in \mathcal{S}_{\sigma,\mu}'$ such that $\mathbb{L} u=F$. Its Fourier coefficients must satisfy
	$$
	\sigma_{L_r}(\tau,j) \widehat{u_j}(\tau) = \widehat{f_{r,j}}(\tau), \quad \text{for all } (\tau,j)\in\mathbb{Z}^m \times \mathbb{N}.
	$$	
	Then, for any $k\in \mathbb{N}$, we have
	\begin{align*}
		|\widehat{u_j}(\tau)| & = |\sigma_{L_r}(\tau,j)|^{-1} |\widehat{f_{r,j}}(\tau)| \geq  \|\sigma_{\mathbb{L}}(\tau,j)\|^{-1} |\widehat{f_{r,j}}(\tau)| \\
		& \geq k \exp\left[\varepsilon (\|\tau_k\|^{1/\sigma} + j_k^{1/(2n\mu)}) \right], 
	\end{align*}
	which implies that $u \notin \mathcal{S}_{\sigma,\mu}'$, leading to a contradiction. Therefore, $\mathbb{L}$ is not $\mathcal{S}_{\sigma,\mu}$-globally solvable.
\end{proof}

\begin{proposition}\label{DC-implies-GS}
	Assume that for all $\varepsilon>0$, there exists $C_\varepsilon>0$ such that 
	\begin{equation}\label{Dioph-cond-DC-implies-GS}
		\|\sigma_{\mathbb{L}}(\tau,j)\| \geq C_\varepsilon \exp\left[-\varepsilon(\|\tau\|^{1/\sigma}+j^{1/(2n\mu)})\right]
	\end{equation}
	for all $(\tau,j)\in\Z^m\times\N$ whenever $\|\sigma_{\mathbb{L}}(\tau,j)\|\neq 0.$ Then $\LL$ is $\mathcal{S}_{\sigma,\mu}$-globally solvable.
\end{proposition}

\begin{proof}
	Let $r^* = r(\tau,j)\in\{1,\ldots,\ell\}$ be the index such that $$|\sigma_{L_{r^*}}(\tau,j)| = \|\sigma_{\mathbb{L}}(\tau,j)\| = \max_{1\leq r\leq \ell} |\sigma_{L_r}(\tau,j)|.$$ 
	
	Now, given any $F=(f_1,\ldots,f_\ell) \in \mathscr{E}_\LL$, define
	\begin{equation}\label{solution_function_solv}
		\widehat{u_j}(\tau) = \left\{
		\begin{array}{ll}
			0, & \mbox{if } \sigma_{L_{r^*}}(\tau,j)=0, \\
			\sigma_{L_{r^*}}(\tau,j)^{-1} \, \widehat{f_{r^*,j}}(\tau), & \mbox{otherwise. }
		\end{array}
		\right.
	\end{equation}
	
	Clearly, $L_{r^*}u=f_{r^*}$. Moreover, since $F\in \mathscr{E}_\LL$, it follows that $L_{s} f_{r^*} = L_{r^*}f_s$, implying $\sigma_{L_{s}}(\tau,j) \widehat{f_{r^*,j}}(\tau) = \sigma_{L_{r^*}}(\tau,j) \widehat{f_{s,j}}(\tau)$. Thus, we have
	\begin{align*}
		\sigma_{L_{s}}(\tau,j)\widehat{u_j}(\tau) & =  \sigma_{L_{s}}(\tau,j)\sigma_{L_{r^*}}(\tau,j)^{-1} \widehat{f_{r^*,j}}(\tau) \\
		& =  \sigma_{L_{r^*}}(\tau,j)^{-1} \sigma_{L_{r^*}}(\tau,j)\widehat{f_{s,j}}(\tau)\\ 
		&= \widehat{f_{s,j}}(\tau),
	\end{align*}		
	for all $1\leq s\leq \ell$ and $(\tau,j)\in\Z^m\times \N$. This gives that $\LL u=F$.	
	Therefore, in order to prove that $\widehat{u_j}(\tau)$ is well-defined and satisfies $L_s u = f_s$, for all $1\leq s\leq \ell$, we only need to show that $\{\widehat{u_j}(\tau)\}$ corresponds to an ultradistribution $u \in \mathcal{S}_{\sigma,\mu}'$.
	
Let $\varepsilon > 0$ be fixed. Since $f_{r^*} \in \mathcal{S}_{\sigma,\mu}'$, there exists a positive constant $C_{\varepsilon,r^*}$ such that
$$
|\widehat{f_{{r^*},j}}(\tau)| \leq C_{\varepsilon,r^*} \exp\left[\frac{\varepsilon}{2} (\|\tau\|^{1/\sigma}+j^{1/(2n\mu)})\right].
$$

By condition \eqref{Dioph-cond-DC-implies-GS}, there exists $C_{\varepsilon} > 0$ such that
$$
\|\sigma_{\mathbb{L}}(\tau,j)\| \geq C_{\varepsilon} \exp\left[-\frac{\varepsilon}{2} (\|\tau\|^{1/\sigma}+j^{1/(2n\mu)})\right].
$$

Hence, we have $ |\widehat{f}_{r^*,j}(\tau)| = |\sigma_{L_{r^*}}(\tau,j)| |\widehat{u_j}(\tau)| = \|\sigma_{\mathbb{L}}(\tau,j)\| |\widehat{u_j}(\tau)|$ and
$$
|\widehat{u_j}(\tau)| = |\widehat{f}_{r^*,j}(\tau)| \|\sigma_{\mathbb{L}}(\tau, j)\|^{-1}  \leq C_{\varepsilon,r^*}C_{\varepsilon}^{-1}\exp\left[\varepsilon (\|\tau\|^{1/\sigma}+j^{1/(2n\mu)})\right].
$$
for all $(\tau,j) \in \mathbb{Z}^n \times \N$, and therefore $u \in \mathcal{S}_{\sigma,\mu}'$, which concludes the proof.
\end{proof}

A more detailed analysis of the last proof shows that it is possible to obtain a better control on the Fourier coefficients of the ultradistribution $u$ in the case where the functions $f_r\in \mathcal{S}_{\sigma,\mu}, \ 1\leq r\leq\ell$. More precisely, we have the following result.

\begin{theorem}
	If $\LL$ is $\mathcal{S}_{\sigma,\mu}$-globally solvable and $F=(f_1,\ldots,f_\ell)\in \mathscr{E}_\LL$ is such that $f_r\in\mathcal{S}_{\sigma,\mu},$ for all $r=1,\ldots,\ell$, then there exists $u\in\mathcal{S}_{\sigma,\mu}$ solution of \ $\LL u=F.$
\end{theorem}

\begin{proof}
	Let $F=(f_1,\ldots,f_\ell)\in \mathscr{E}_\LL$ such that $f_r\in\mathcal{S}_{\sigma,\mu}$ for all $r=1,\ldots,\ell$.  By the definition of the spaces $\mathcal{S}_{\sigma,\mu}$, there are positive constants $\varepsilon$ and $C$ such that
	\begin{equation*}
		|\widehat{f_{r,j}}(\tau)| \leq C \exp\left[-\varepsilon(\|\tau\|^{1/\sigma}+j^{1/(2n\mu)}) \right].
	\end{equation*}
	for all $r=1,\ldots,\ell$, and $(\tau,j)\in\Z^m\times\N$. 

	Let $u$ be defined as in \eqref{solution_function_solv}, that is,
	\begin{equation*}
		\widehat{u_j}(\tau) = \left\{
		\begin{array}{ll}
			0, & \mbox{if } \sigma_{L_{r^*}}(\tau,j)=0, \\
			\sigma_{L_{r^*}}(\tau,j)^{-1} \, \widehat{f_{r^*,j}}(\tau), & \mbox{otherwise. }
		\end{array}
		\right.
	\end{equation*}
	where $r^* = r(\tau,j)\in\{1,\ldots,\ell\}$ is the index such that $|\sigma_{L_{r^*}}(\tau,j)| = \|\sigma_{\mathbb{L}}(\tau,j)\|.$ We already proved that $\LL u=F$. To conclude the proof we just need to show that $u \in \mathcal{S}_{\sigma,\mu}.$

	Since $\LL$ is $\mathcal{S}_{\sigma, \mu}$-globally solvable, \eqref{Dioph-cond-GS-implies-DC} holds. Therefore, choosing $\varepsilon>\varepsilon_{r^*}$, there exists $C_\varepsilon>0$ such that
	\begin{align*}
		|\widehat{u_j}(\tau)| & = |\sigma_{L_{r^*}}(\tau,j)|^{-1} |\widehat{f_{r^*,j}}(\tau)| \\
		& = \|\sigma_{\LL}(\tau,j)\|^{-1} |\widehat{f_{r,j}}(\tau)| \\
		& \leq C_\varepsilon^{-1} \exp\left[\varepsilon(\|\tau\|^{1/\sigma} + j^{1/(2n\mu)}) \right] |\widehat{f_{r,j}}(\tau)|,
	\end{align*}
	for all $(\tau,j)\in \Z^m\times\N$.
	
	This leads to:
	\begin{align*}
		|\widehat{u_j}(\tau)| & \leq C_\varepsilon^{-1} C_0 \exp\left[-(\varepsilon_0-\varepsilon)(\|\tau\|^{1/\sigma}+j^{1/(2n\mu)}) \right] 
	\end{align*}
	thus $u\in\mathcal{S}_{\sigma,\mu}.$
\end{proof}

\begin{remark}
	The last result says that we can obtain a strong solution for $\LL u = F$ in the case where $\LL $ is globally solvable and $F$ is a vector-valued admissible function. However, it is important to note that this does not imply that $\LL$ is globally hypoelliptic.
\end{remark}

\begin{remark}
	We note that when $m=\ell =1$ and $Q_1(D_t) = D_t$, the system \eqref{general-const-system} is reduced to a single operator given by
	$$
	\mathcal{L} = D_t+ d P(x,D_x), 
	$$
	where $d\in\C$ and the symbol is given by
	$$
	\sigma_{\mathcal{L}}(\tau,j) = \tau + d\lambda_j, \quad (\tau,j) \in \Z\times \N.
	$$
	
	This particular case is included in the class of operators studied in \cite{AviCap22,AviCap24}. 
	Notably, the condition
	$$
	|\tau + d\lambda_j| \geq C_\varepsilon \exp\left[-\varepsilon(\|\tau\|^{1/\sigma}+j^{1/(2n\mu)})\right], \quad (\tau,j)\in\Z\times\N, 
	$$
	is equivalent to 
	$$
	\inf_{\tau \in \Z}|\tau + d\lambda_j| \geq C_\varepsilon
	\exp\left(-\varepsilon j^{1/(2n\mu)}\right), \quad j\to \infty.
	$$ 
	
	Therefore, our results recover the case of a single operator with constant coefficients considered in these references.
\end{remark}

\begin{remark}
	Since the conditions determining the global hypoellipticity and global solvability properties of the time-independent system of type \eqref{general-const-system} rely primarily on the behavior of the the symbol $\sigma_{\LL}$ at infinity, our results can be easily extended to systems of periodic pseudo-differential operators. Specifically, the differential operators $Q_r(D_t), 1\leq r\leq \ell,$ treated in Theorems \ref{GH-thm} and \ref{GS-thm} can be substituted with periodic pseudo-differential operators defined on the torus $\TT^m$, namely, those operators formally defined by 
	\begin{equation*}
		a_r(D_t) u(t) =  \sum_{\eta \in \mathbb{Z}^m}{e^{i t \cdot \eta} a_r(\eta) \widehat{u}(\eta)},
	\end{equation*}
	with symbol $a(\eta) = \{a(\eta)\}_{\eta \in \mathbb{Z}^m}$ satisfying some moderate growth like
	\begin{equation*}
		|a_r(\eta)| \leq C |\eta|^{\nu}, \ \forall \eta \in \mathbb{Z}^{m}.
	\end{equation*}
\end{remark}

\section{Examples}

To illustrate our results, we present examples of systems in this section. We begin by introducing the notion of exponential Liouville vectors.

\begin{definition}
	A vector $\alpha \in \mathbb{R}^m\setminus\mathbb{Q}^m$ is said to be an exponential Liouville vector of order $\sigma \geq 1$ if there exists $\epsilon>0$ such that the inequality 
	$$
	\|\tau - \ell \alpha\| \leq \exp\left(-\epsilon |\ell|^{1/\sigma}\right)
	$$
	has infinitely many solutions $(\tau, \ell) \in \mathbb{Z}^m\times \mathbb{Z}$.
\end{definition}

It follows that $\alpha \in \mathbb{R}^m\setminus\mathbb{Q}^m$ is not an exponential Liouville vector of order $\sigma \geq 1$ if, for every $\epsilon>0$, there exists $C_{\epsilon}>0$ such that 
$$
\|\tau - \ell \alpha\| \geq \exp\left(-\epsilon |\ell|^{1/\sigma}\right), 
$$
for all $(\tau, \ell) \in \mathbb{Z}^m\times \mathbb{Z}$.

It is worth noting that when $m=1$, this notion aligns with the concept of Liouville number. An example of a non-exponential Liouville number of order $\sigma\geq 1$ is the continued fraction $\omega =  [10^{1!}, 10^{2!}, 10^{3!}, \dots]$ (see \cite{AriKirMed19}).

Furthermore, if a single coordinate $\alpha_i$ of a vector $\alpha \in \mathbb{R}^m\setminus\mathbb{Q}^m$ is not an exponential Liouville number, then the entire vector shares the same property. However, the converse is not true. For example, in Example 4.9 of \cite{Ber99} (also see Example 2.7 in \cite{BerMedZan21}), a pair of exponential Liouville numbers $\rho$ and $\theta$ is constructed such that the vector $(\rho, \theta)$ is not an exponential Liouville vector of order $\sigma\geq1$.

With this discussion in mind, taking a non-exponential Liouville vector  $(\rho, \theta)$ whose coordinates are Lioville numbers, consider the following systems:
$$
\mathbb{L} \doteq    
\begin{cases}
	L_1 = D_{t_1} + i\rho  H  \\
	L_2 = D_{t_2} + \theta H    
\end{cases}
\quad \text{and} \quad
\mathbb{M} \doteq    
\begin{cases}
	M_1 = D_{t_1} + \rho H  \\
	M_2 = D_{t_2} + \theta H
\end{cases}
$$
defined on $\mathbb{T}^2 \times \mathbb{R}$, where $H$ represents the harmonic oscillator given by
\begin{equation*}
	H = -\frac{d^2}{dx^2} + x^2, \quad x \in \mathbb{R},
\end{equation*}    
with eigenvalues $\lambda_j = 2j + 1$, $j \in \mathbb{N}$. 

Since $L_1$ is globally hypoelliptic, then the same is true for the system $\mathbb{L}$ (and also globally solvable). Note that the operator
$L_2$ is neither globally hypoelliptic nor globally solvable.

On the other hand, the system $\mathbb{M}$ is globally hypoelliptic, and therefore globally solvable, although operators $M_1$ and $M_2$ do not have these properties.

\section{A class of time dependent coefficients systems} \label{sec_reduction}

In this section, we introduce a class of operators which depend on time and whose global hypoellipticity and solvability can be characterized by their normal form. This normal form is a time-independent system uniquely associated with the original system through conjugation.

We begin by considering the system of operators $\LL_a = (L_1, L_2,\ldots,L_m)$ defined on $\mathbb{T}^m \times \mathbb{R}^n$, given by
\begin{equation}\label{L_a}
	L_r = D_{r} + a_r(t_r)P(x,D_x), \quad r=1, \ldots, m,
\end{equation}
where each $a_{r}\in \mathcal{G}^{\sigma}(\mathbb{T}^1; \mathbb{R})$ is a real-valued function, the derivatives $D_r=-i\partial/\partial t_r$ are defined in different copies of the one-dimensional torus, and $P = P(x,D_x)$ satisfies the conditions \eqref{P-intro} and \eqref{P-elliptic}.

Additionally, we define the normal form of $\LL_a$ as the system $\LL_{a_0} = (L_{1,0}, L_{2,0},\ldots, L_{m,0})$ in $\mathbb{T}^m \times \mathbb{R}^n$, given by
\begin{equation}\label{L_0}
	L_{r,0} = D_r + a_{r,0}P(x,D_x), \quad r=1, \ldots, m,
\end{equation}
where
$$
a_{r,0} = (2 \pi)^{-1} \int_{0}^{2\pi} a_r(\eta)d\eta.
$$
This technique of obtaining global properties of $\LL_a$ by analyzing the corresponding properties of the normal form $\LL_{a_0}$ is widely employed for systems and operators defined on the torus, as seen in \cite{AriKirMed19, BerMedZan21, Pet05}. 
The idea is to find a linear automorphism $\Psi$ of the spaces $\mathcal{S}_{\sigma, \mu}$ and $\mathcal{S}'_{\sigma, \mu}$ which works as a bijection between the spaces of admissible distributions of the operators $\LL_a$ and $\LL_{a_0}$ and such that
\begin{equation}\label{conjugationformula}
	\Psi^{-1} \circ L_r\circ \Psi =  L_{r,0}, \quad 1\leq r \leq m.
\end{equation}
Indeed, assuming the existence of such an automorphism, if ${\LL}_{a_0}$ is $\mathcal{S}_{\sigma,\mu}$-globally hypoelliptic and $u \in \mathcal{S}_{\sigma,\mu}'$ satisfy $L_{r} u = f_r \in \mathcal{S}_{\sigma,\mu}$ for each $r=1, \ldots, m$. We may then write 
$$
f_r = \Psi g_{r}, \text{ with } g_r \in \mathcal{S}_{\sigma,\mu}, \text{ for all } r=1, \ldots, m.
$$
Then, by \eqref{conjugationformula}, we have
$$
L_{r} u = \Psi g_{r} \Longrightarrow L_{r,0}(\Psi^{-1} u) = g_r, \quad r=1, \ldots, m.
$$

Since $\Psi^{-1} u \in \mathcal{S}_{\sigma,\mu}'$, it follows from the hypothesis on ${\LL}_{a_0}$ that $\Psi^{-1} u \in \mathcal{S}_{\sigma,\mu}$, implying $u \in \mathcal{S}_{\sigma,\mu}$. Therefore, ${\LL}_{a}$ is $\mathcal{S}_{\sigma,\mu}$-globally hypoelliptic. Similarly, the $\mathcal{S}_{\sigma,\mu}$-global hypoellipticity of ${\LL}_{a}$ implies the same property for ${\LL}_{a_0}$, using a similar argument.

Now, suppose that $\mathbb{L}_{a_0}$ is $\mathcal{S}_{\sigma,\mu}$-globally solvable, and let 
$$
F = (f_1, \ldots, f_m) \in \mathscr{E}_{{\LL}_{a}}.
$$
Since $f_r =  \Psi g_r$ for some $G = (g_1, \ldots, g_m) \in \mathscr{E}_{{\LL}_{0}}$, there exists
$u \in \mathcal{S}_{\sigma,\mu}'$ such that $L_{r,0}u = g_r$. Then, 
according to the conjugation formula, 
$$
f_r = \Psi(L_{r,0} u ) = L_{r} (\Psi u),  \quad r=1, \ldots, m.
$$
Thus, $v = \Psi u  \in \mathcal{S}_{\sigma,\mu}'$ is a solution of
$L_{r} v = f_r$, implying the $\mathcal{S}_{\sigma,\mu}$-global solvability of $\mathbb{L}_{a}$. The converse follows a similar argument.
  For the choice of $\Psi$ we shall take ispiration from a similar result proved in \cite{AviCap24}. Unluckily, the map we have in mind maps isomorphically $\mathcal{S}_{\sigma,\mu}$ into $\mathcal{S}_{\tilde{\sigma}, \mu}$ where $\tilde{\sigma} = \max \{\sigma, \mu M-1\},$ so it yields a loss of regularity in the time variables with an interplay between the regularity in $t$ and the one in $x$. This is not surprising because such a phenomenon had been already observed in \cite{AviCap22} for the case of time-depending coefficients with the same loss. This obstacle cannot be removed even by replacing $\mathcal{S}_{\sigma,\mu}$ and $\mathcal{S}'_{\sigma,\mu}$ by the sets
  $$\mathcal{F}_\mu = \bigcup_{\sigma \geq 1}\mathcal{S}_{\sigma,\mu}, \qquad \mathcal{U}_\mu = \bigcup_{\sigma \geq 1}\mathcal{S}'_{\sigma,\mu}$$
  in the definition of global hypoellipticity and solvability since it will be clear from the next arguments that the function $\Psi$ does not map $\mathcal{U}_\mu $ into itself.

In view of these considerations, in this section we restrict our analysis to the case
$$
M\mu - 1 \leq \sigma.
$$

Nevertheless, despite this restriction, we stress the fact that we are able to explore the noteworthy case where $\sigma = 1$, $M = 2$, and $1/2 \leq \mu < 1$, yielding the space $\mathcal{S}_{1, \mu}(\mathbb{T}^m \times \mathbb{R}^n)$ of functions that have an entire extension to $\mathbb{T}^m \times \mathbb{C}^n$, satisfying
\[
|\partial_z^\beta \partial_t^{\alpha}\tilde f(t,z) | \leq C^{|\beta|+|\alpha|+1} \alpha! \beta!^\mu \exp \left(-a |\operatorname{Re} z |^{\frac{1}{\mu}}+ b |\operatorname{Im} z|^{\frac{1}{1-\mu}}\right),
\]
where the positive constants $C$, $a$, and $b$ are independent of $\alpha$ and $\beta$. We are not aware of results on global hypoellipticity or solvability in spaces of entire functions on $\C^n$ in the literature. 

Now, based on the preceding sections and the results from \cite{AviCap24}, we define the set $\mathscr{E}_{\mathbb{L}_a}$ of admissible functions for the global solvability of the system $\mathbb{L}_a$. This set is formed by all vectors $F=(f_1, \ldots,f_m)$ with coordinates $f_r \in \mathcal{S}_{\sigma,\mu}$ that satisfy
\begin{equation}\label{commutation-a}
	L_r f_s = L_s f_r, \quad 1 \leq r,s \leq m,
\end{equation}
and
\begin{equation}\label{compatib-a}
	\int_{0}^{2\pi} \exp\left(i \lambda_j \int_{0}^{t_r}a_r(\eta)d\eta\right) f_{r,j}(t_1, \ldots, t_r, \ldots, t_{m}) dt_{r} = 0,
\end{equation}
whenever $a_{r,0} \lambda_j \in \mathbb{Z}$.

Similarly, $\mathscr{E}_{\mathbb{L}{a_0}}$ represents the set of all vectors $F=(f_1, \ldots,f_m)$, where each $f_r \in \mathcal{S}_{\sigma,\mu}$ satisfies:
\begin{equation}\label{commutation-0}
	L_{r,0} f_s = L_{s,0} f_r, \quad 1 \leq r,s \leq m,
\end{equation}
and
\begin{equation}\label{compatib-0}
	\widehat{f}_{r,j}(\tau) = 0, \quad \text{whenever} \quad
	\sigma_{L_{r,0}}(\tau,j) = \tau_r + a_{r,0}\lambda_j = 0.
\end{equation}

Finally, we recall the notions of hypoellipticity and solvability,

\begin{definition}
	Let $\LL_{a}$ and $\LL_{a_0}$ be the systems defined in \eqref{L_a} and \eqref{L_0}. We say that:
	\begin{enumerate}
		\item $\LL_a$ (resp. $\LL_{a_0}$) is $\mathcal{S}_{\sigma,\mu}$-globally solvable when  $\LL_a(\mathcal{S}_{\sigma,\mu}')=\mathscr{E}_{\LL_a}$ (resp. $\LL_{a_0}(\mathcal{S}_{\sigma,\mu}')=\mathscr{E}_{\LL_{a_0}}$); 
		\item $\LL_a$ (resp. $\LL_{a_0}$) is $\mathcal{S}_{\sigma,\mu}$-globally hypoelliptic when $u \in \mathcal{S}_{\sigma,\mu}'$  and $L_r u \in \mathcal{S}_{\sigma,\mu}$ (resp. \\ $L_{r,0} u \in \mathcal{S}_{\sigma,\mu}$), for all $1\leq r\leq m$, imply  $u \in \mathcal{S}_{\sigma,\mu}.$
	\end{enumerate}
\end{definition}

Now, we present our main result regarding the global properties of these systems.

\begin{theorem}\label{The-Normal} Let $\sigma \geq 1$ and $\mu \geq 1/2$ such that $\mu M -1 \leq \sigma$. Then
	the system $\mathbb{L}_{a}$ is $\mathcal{S}_{\sigma,\mu}$-globally hypoelliptic (resp. $\mathcal{S}_{\sigma,\mu}$-globally solvable) if and only if its normal form $\mathbb{L}_{a_0}$ is $\mathcal{S}_{\sigma,\mu}$-globally hypoelliptic (resp. $\mathcal{S}_{\sigma,\mu}$-globally solvable).
\end{theorem}

To prove the theorem, we need to establish the existence of an automorphism that satisfies \eqref{conjugationformula}. Initially, we demonstrate the well-definedness, linearity, and invertibility of $\Psi$. 

\begin{proposition}\label{Theorem-Psi}
	For each ultradistribution $u(t,x)=\sum_{j \in \N} u_j(t) \varphi_j(x) \in \mathcal{S}_{\sigma,\mu}'$ let
	\begin{equation}\label{conjugation}
		(\Psi u)(t,x) \doteq \sum_{j \in \N} \exp\left(-i A(t) \lambda_j\right) u_j(t)\varphi_j(x),
	\end{equation}
	where
	\begin{equation}\label{A(t)definition}
	A(t) = \sum_{k=1}^{m} \int_{0}^{t_k} a_k(\eta)d\eta - a_{k,0}t_k. 
	\end{equation}
	Then, for all $\sigma \geq 1$ and $\mu \geq 1/2$ such that $M\mu - 1 \leq \sigma,$ we have that
	$$
	\Psi: \mathcal{S}_{\sigma,\mu} \to \mathcal{S}_{\sigma,\mu}
	\ \textrm{ and } \
	\Psi: \mathcal{S}_{\sigma,\mu}' \to \mathcal{S}_{\sigma,\mu}', 
	$$
	are automorphisms with inverse
	\begin{equation}\label{conjugation_inverse}
		\Psi^{-1} u(t,x) = \sum_{j \in \N} \exp \left(i A(t) \lambda_j\right) u_j(t)\varphi_j(x).
	\end{equation}	
	
\end{proposition}

\begin{proof}
Let us show that $\Psi: \mathcal{S}_{\sigma,\mu} \to \mathcal{S}_{\sigma,\mu}$ is well-defined. 
Let $f = \sum_{j \in \mathbb{N}} f_j(t)\varphi_j(x)\in \mathcal{S}_{\sigma,\mu}$, and fix $\alpha \in \Z^m_+$.
Denoting $\psi_j(t) = \exp\left(-i A(t) \lambda_j\right) f_j(t)$, by Leibniz formula, 
$$
|\partial_t^{\alpha}\psi_j(t)| \leq \sum_{\beta \leq \alpha}
\binom{\alpha}{\beta} 
|\partial_t^{\beta}\exp \left(i A(t) \lambda_j\right)|
|\partial_t^{\alpha-\beta} f_j(t)|.
$$

By Theorem \ref{Thm-estimate-derivatives-functions}, there exist $\epsilon, C > 0$ such that
\begin{equation*}
	\sup_{t \in \mathbb{T}^m} | \partial_t^{\alpha-\beta} f_j(t)| \leq
	C^{|\alpha-\beta|+1} [(\alpha-\beta)!]^{\sigma} \exp \left[-\varepsilon j^{\frac{1}{2n\mu}} \right].
\end{equation*}    

Note that
$$
|\partial_t^{\beta}\exp \left(i A(t) \lambda_j\right)| = 
\prod_{k=1}^{m} \left| \partial_{t_k}^{\beta_k} \exp \left[-i \lambda_j \left(\int_{0}^{t_k} a_k(\eta)d\eta - a_{k,0}t_k\right)\right]\right|,
$$
where $\beta = (\beta_1, \ldots, \beta_m)$. Denoting
$$
\mathcal{H}_j(t_k) = \exp\left[-i\lambda_j \left( \int_{0}^{t_k} a_k(\eta)\, d\eta  - a_{k,0}t_k\right)\right],
$$
we obtain from Fa\`a di Bruno's formula
$$
\partial_{t_k}^{\beta_k} \mathcal{H}_j(t_k)= \sum_{\Delta(\gamma), \, \beta_k}
\dfrac{1}{\gamma!}(-i\lambda_j)^\gamma
\frac{\beta_k!}{s_1! \cdots s_\gamma! }
\left(\prod_{\nu=1}^\gamma
\partial_{t_k}^{s_\nu-1}(a_k(t_k)-a_{k,0}) \right)
\mathcal{H}_j(t_k),
$$
where
$
\sum\limits_{\Delta(\gamma), \, \beta_k} = \sum\limits_{\gamma=1}^{\beta_k}\sum\limits_{\stackrel{s_1+\ldots+s_\gamma=\beta_k}{s_\nu \geq 1, \forall \nu}}.
$

Since
$$
\left| \prod_{\nu=1}^\gamma
\partial_{t_k}^{s_\nu-1}(a_k(t_k)-a_{k,0})  \right| \leq C^{\beta_k-\gamma+1}[(\beta_k -\gamma)!]^\sigma,
$$
and $|\lambda_j| \leq C_1 j^{M /2n}$, we get
\begin{align*}
	|\partial_{t_k}^{\beta_k} \mathcal{H}_j(t_k)| & \leq 
	\sum_{\Delta(\gamma), \, \beta_k}|\lambda_j|^\gamma
	\dfrac{1}{\gamma!}
	\frac{\beta_k!}{s_1! \cdots s_\gamma! }
	C^{\beta_k-\gamma+1}[(\beta_k -\gamma)!]^\sigma \\
	& \leq 
	\sum_{\Delta(\gamma), \, \beta_k}C_1^{\gamma}j^{\frac{\gamma M}{2n}}
	\dfrac{1}{\gamma!} 
	\frac{\beta_k!}{s_1! \cdots s_\gamma! }
	C^{\beta_k-\gamma+1}[(\beta_k -\gamma)!]^\sigma \\
\end{align*}

In view of
\begin{equation*}
	j^{\frac{\gamma M}{2n}} 
	\leq C_{\epsilon}^{\gamma}(\gamma !)^{M\mu} \exp \left(\epsilon/(2m) j^{\frac{1}{2n\mu}} \right),
\end{equation*}
it follows that
$$
|\partial_{t_k}^{\beta_k} \mathcal{H}_j(t_k)| \leq C_2^{\beta_k + 1} \exp \left(\epsilon/(2m)j^{\frac{1}{2n\mu}} \right) (\beta_k !)^{\sigma},
$$
since we have assumed $M\mu - 1 \leq \sigma$. 

Therefore, 
\begin{equation}\label{est_e_A}
	|\partial_t^{\beta}\exp \left(i A(t) \lambda_j\right)| \leq 
	C_2^{|\beta|+ 1}
	\exp \left(\epsilon/2 \, j^{\frac{1}{2n\mu}} \right) (\beta!)^{\sigma}
\end{equation}
implying
$$
|\partial^{\alpha}_t \psi_j(t)|\leq C^{|\alpha| +1} (\alpha !)^{\sigma} 
\exp \left(-\epsilon/2 \, j^{\frac{1}{2n\mu}} \right),
$$	 
and consequently, $\Psi f \in  \mathcal{S}_{\sigma,\mu}$.

\bigskip	
Now, let us show that $\Psi: \mathcal{S}_{\sigma,\mu}' \to \mathcal{S}_{\sigma,\mu}'$ is well-defined. Given $u \in \mathcal{S}_{\sigma,\mu}'$, we need to verify that $\Psi u \in \mathcal{S}_{\sigma,\mu}'$. To do this, it is sufficient to prove that for every $\epsilon,h>0$, there is $C_{\epsilon,h} >0$ satisfying
\begin{equation*}
	|\langle  \psi_j(t) \, , \, \theta(t) \rangle | \leq C_{\epsilon,h} \|\theta\|_{\sigma, h} \exp \left(\epsilon j^{\frac{1}{2n\mu}}\right),
\end{equation*}
for all $j \in \mathbb{N}$, and for all $\theta \in \mathcal{G}^{\sigma,h}(\TT^m)$, cf. \cite{AviCap22}.

Let $\epsilon, h>0$ be fixed. Given $\theta \in \mathcal{G}^{\sigma,h}(\TT^m)$, it follows from \eqref{est_e_A} that
\begin{align*}
	|\partial_t^{\alpha}\left[\exp\left(-i A(t) \lambda_j\right)\theta(t)\right]| & \leq   \exp\left(\epsilon/2 j^{\frac{1}{2n\mu}} \right)
	\sum_{\beta \leq \alpha} \binom{\alpha}{\beta}  C^{|\beta|+ 1}
	(\beta!)^{\sigma} |\partial_t^{\alpha -\beta}\theta(t)| \\
	& \leq  \exp\left(\epsilon/2 j^{\frac{1}{2n\mu}} \right) \sum_{\beta \leq \alpha} \binom{\alpha}{\beta}  C^{|\beta|+ 1}
	(\beta!)^{\sigma} \|\theta\|_{\sigma,h}((\alpha - \beta)!)^{\sigma}h^{|\alpha-\beta|} \\
	& \leq  \exp\left(\epsilon/2 j^{\frac{1}{2n\mu}} \right) C_{\epsilon,h}^{|\alpha|} (\alpha!)^{\sigma} \|\theta\|_{\sigma,h}.
\end{align*} 

By the characterization of $\mathcal{S}_{\sigma,\mu}$ in terms of semi-norms \eqref{secondnorm}, we have the following: for every $A>0$, there is $C_A>0$ such that
\begin{equation*}
	|\langle  u\, , \, \Theta(t,x) \rangle | 
	\leq C_A \sup\limits_{\substack{\alpha \in \N^m \\ k \in \N}}    
	A^{-k - |\alpha|} (k!)^{-M \mu} (\alpha!)^{-\sigma} \|P^k \partial_t^{\alpha}\Theta(t,x)\|_{L^2(\T^m\times \R^n)},
\end{equation*}
for all $\Theta \in \mathcal{S}_{\sigma,\mu}$. Hence, setting
$$
\Omega_{j}(t)=\langle  \psi_j(t) \, , \, \theta(t) \rangle = \langle  u \, , \, \exp\left(-i\lambda_j A(t) \right)\theta(t)\varphi_j(x) \rangle,
$$
we obtain
	\begin{align*}
		|\Omega_{j}(t)| 
		& \leq 
		C_{A} \sup\limits_{\stackrel{\alpha \in \Z^m_+}{k \in \N}}	
		A^{-k - |\alpha|} (k!)^{-M \mu} (\alpha!)^{-\sigma} \|P^k \partial_t^{\alpha}\left[\exp\left(-i A(t) \lambda_j\right)\theta(t)\varphi_j(x) \right]\|_{L^2(\T^m\times \R^n)}\\
		& = 
		C_{A} \sup\limits_{\stackrel{\alpha \in \Z^m_+}{k \in \N}}	
		A^{-k - |\alpha|} (k!)^{-M \mu} (\alpha!)^{-\sigma}  |\partial_t^{\alpha}\left[\exp\left(-i A(t) \lambda_j\right)\theta(t)\right]| \|P^k\varphi_j(x) \|_{L^2(\T^m\times \R^n)} \\
		& =
		C_{A} \sup\limits_{\stackrel{\alpha \in \Z^m_+}{k \in \N}}	
		A^{-k - |\alpha|} (k!)^{-M \mu} (\alpha!)^{-\sigma}  |\partial_t^{\alpha}\left[\exp\left(-i A(t) \lambda_j\right)\theta(t)\right]| |\lambda_j|^k \\
		& \leq
		C_{A} \sup\limits_{\stackrel{\alpha \in \Z^m_+}{k \in \N}}	
		A^{-k - |\alpha|} (k!)^{-M \mu} (\alpha!)^{-\sigma}   C^{k}j^{kM/2n } |\partial_t^{\alpha}\left[\exp\left(-i A(t) \lambda_j\right)\theta(t)\right]|  \\
		& \leq
		C_{A} \sup\limits_{\stackrel{\alpha \in \Z^m_+}{k \in \N}}	
		A^{-k - |\alpha|} (k!)^{-M \mu} (\alpha!)^{-\sigma}   C^{k} C_{\epsilon}^{k}(k!)^{M\mu} \exp\left(\epsilon j^{\frac{1}{2n\mu}} \right) C_{\epsilon,h}^{|\alpha|} (\alpha!)^{\sigma} \|\theta\|_{\sigma,h}  \\
		& = C_{A} \|\theta\|_{\sigma,h} \exp\left(\epsilon j^{\frac{1}{2n\mu}} \right) \sup\limits_{\stackrel{\alpha \in \Z^m_+}{k \in \N}} A^{-k - |\alpha|}  C_{\epsilon,h}^{|\alpha|} C_{\epsilon}^{k}.
	\end{align*}
	
Then, for $A^{-1} = \max\{C_{\epsilon}, C_{\epsilon,h}\}$,  
$$
|\langle  \psi_j(t) \, , \, \theta(t) \rangle|\leq \widetilde{C}_{\epsilon,h}  \|\theta\|_{\sigma,h} \exp\left(\epsilon j^{\frac{1}{2n\mu}} \right).
$$
To finish the proof, it is easy to verify that $\Psi^{-1}$ is the inverse of $\Psi$, and that both are linear, which concludes the proof.
\end{proof}

From expressions \eqref{conjugation} and \eqref{conjugation_inverse}, which define the expressions of $\Psi$ and its inverse, it is clear that $L_r \circ \Psi = \Psi \circ L_{r,0}$ for every $1\leq r \leq m$. In other words, the automorphism $\Psi$ satisfies the property stated in \eqref{conjugationformula}.

To conclude the proof of global solvability, we only need to prove that the sets of functions admissible for solvability are isomorphic under the action of the map $\Psi$.

\begin{proposition}\label{Theorem-Reeduc}
	The map 
	$$
	(f_1, \ldots, f_m)\in \mathscr{E}_{{\LL}_{a_0}}  \longmapsto (\Psi f_1, \ldots, \Psi f_m) \in \mathscr{E}_{{\LL}_{a}},
	$$
	is an isomorphism.
\end{proposition}

\begin{proof} 
	Given $(f_1, \ldots, f_m) \in \mathscr{E}_{{\LL}_{0}}$ and $(g_1, \ldots, g_m) \in \mathscr{E}_{{\LL}_{a}}$,  we have
	\begin{align*}
		&L_r(\Psi f_s) = \Psi (L_{r,0} f_s) = \Psi (L_{s,0} f_r) = L_s(\Psi f_r), \text{ and} \\
		&L_{r,0}(\Psi^{-1} g_s) = \Psi^{-1} (L_r g_s) = \Psi^{-1}(L_{s} g_r) = L_{s,0}(\Psi^{-1} g_r),		
	\end{align*}
	for any $r,s\in\{1,\ldots,m\}$. Therefore, the properties \eqref{commutation-a} and \eqref{commutation-0} are preserved by $\Psi$. 
	
	Next, we prove that \eqref{compatib-a} and \eqref{compatib-0} are also preserved by $\Psi$. For $(f_1, \ldots, f_m) \in \mathscr{E}_{{\LL}_{0}}$, we define
	$$
	\psi_{r,j}(t) = \exp \left(-i A(t) \lambda_j\right) f_{r,j}(t),
	$$
	and 
	$$
	\Theta_{r,j}(t) =  	\int_{0}^{2\pi} \exp\left(i \lambda_j \int_{0}^{t_r}a_r(\eta)d\eta\right) \psi_{r,j}(t) dt_{r},
	$$
	for $1\leq r\leq m$, $j\in\N$, and $t=(t_1,\ldots,t_m)\in\TT^m$. We recall that the function $A$ was defined in \eqref{A(t)definition}.

	Observe that
	\begin{align*}
		\Theta_{r,j}(t) &= \sum_{\tau \in\Z^m}  \exp\left( \sum_{k \neq r} \int_{0}^{t_k} a_k(\eta)d\eta - a_{k,0}t_k\right)\widehat{f_{r,j}}(\tau) \int_{0}^{2\pi}\exp \left[it_{r} (\lambda_j a_{r,0} + \tau_r) \right] dt_{r}.
	\end{align*}

	Assume $a_{r,0} \lambda_j \in \mathbb{Z}$. If $\lambda_j a_{r,0} + \tau_r = 0$, then by \eqref{compatib-0}, we have $\widehat{f_{r,j}}(\tau) = 0$, and consequently, $\Theta_{r,j}(t) = 0$. 

	For $\lambda_j a_{r,0} + \tau_r \neq  0$, we have
	$$
	\int_{0}^{2\pi}\exp \left[it_{r} (\lambda_j a_{r,0} + \tau_r) \right] dt_{r} = 0,
	$$	
	implying $\Theta_{r,j}(t) = 0$. Therefore, $\Psi f_r$ satisfies \eqref{compatib-a}, implying $(\Psi f_1, \ldots, \Psi f_m) \in  \mathscr{E}_{{\LL}_{a}}$.

	Conversely, for $(g_1, \ldots, g_m) \in  \mathscr{E}_{{\LL}_{a}}$, define
	$$
		\Gamma_{r,j}(t) = \exp \left(i A(t) \lambda_j\right) g_{r,j}(t),
	$$
	for $1\leq r\leq m$, $j\in\N$, and $t=(t_1,\ldots,t_m)\in\TT^m$. 

	Observe that
	\begin{align*}
		\widehat{\Gamma_{r,j}}(\tau) & = \dfrac{1}{(2 \pi)^{m}} \int_{\TT^m} \prod_{k=1}^{m}\left\{	\exp\left(i \lambda_j \int_{0}^{t_k}a_k(\zeta)d\zeta\right) \exp(-i\sigma_{L_{r,0}}(\tau)t_{k})\right\}g_{r,j}(t)dt \\
		& = \dfrac{1}{(2 \pi)^m} \int\limits_{\TT^{m-1}} H_{k,j} \left\{\int_{\TT_r} \exp\left(i \lambda_j \int_{0}^{t_r}a_r(\zeta)d\zeta\right) \exp(-i\sigma_{L_{r,0}}(\tau)t_{r})g_{r,j}(t)dt_r\right\}dt_{-r},
\end{align*}
where $dt_{-r}$ denotes integration with respect to all variables $t_k$ except $t_r$, and
$$
H_{k,j} = 
\prod_{k\neq r}\left\{
\exp\left(i \lambda_j \int_{0}^{t_k}a_k(\zeta)d\zeta\right)
\exp(-i\sigma_{L_{r,0}}(\tau)t_{k})\right\}.
$$	

Hence, $\sigma_{L_{r,0}}(\tau) = 0$ implies  $\widehat{\Gamma_{r,j}}(\tau)=0$ since 
$$
\int_{\TT_r} \exp\left(i \lambda_j \int_{0}^{t_r}a_r(\zeta)d\zeta\right) g_{r,j}(t)dt_r = 0,
$$
and consequently, $\Psi^{-1}g_{r,j}$ satisfies \eqref{compatib-0}.
\end{proof}

\section*{Acknowledgements}

	The first and third authors thank the support provided by the National Council for Scientific and Technological Development - CNPq, Brazil (grants 316850/2021-7, 423458/2021-3, 402159/2022-5, 305630/2022-9, and 200295/2023-3). Additionally, they express their gratitude for the hospitality extended to them during their visit to the Department of Mathematics at the University of Turin, Italy, where most of this work was developed. The second author is partially supported by the INDAM-GNAMPA project CUP E53C23001670001.

\end{document}